\documentclass[11pt]{article}
\usepackage[a4paper,margin=1in]{geometry}

\usepackage{amsmath,amssymb,amsthm}
\usepackage{graphicx}
\usepackage{enumitem}
\usepackage{booktabs}
\usepackage{xcolor}
\usepackage{bbm}
\usepackage{hyperref}
\usepackage[capitalise]{cleveref}
\hypersetup{hidelinks}

\numberwithin{equation}{section}
\counterwithin{figure}{section}

\theoremstyle{plain}
\newtheorem{theorem}{Theorem}[section]
\newtheorem{lemma}[theorem]{Lemma}

\newtheorem{cor}[theorem]{Corollary}

\newtheorem{conj}[theorem]{Conjecture}
\newtheorem{claim}[theorem]{Claim}

\theoremstyle{definition}
\newtheorem{definition}[theorem]{Definition}
\newtheorem{problem}[theorem]{Problem}
\newtheorem{question}[theorem]{Question}

\newtheorem{remark}[theorem]{Remark}

\newcommand{\ex}{\mathrm{ex}}

\title{Minimum-degree thresholds beyond bounded colourability}

\author{
Xinqi Huang\thanks{School of Mathematical Sciences, University of Science and Technology of China, Hefei, China and Extremal Combinatorics and Probability Group (ECOPRO), Institute for Basic Science (IBS), Daejeon, South Korea.
Email: \texttt{huangxq@mail.ustc.edu.cn}. Supported by the USTC Excellent PhD Students Overseas, the Institute for Basic Science (IBS-R029-C4), the National Key Research and Development Programs of China 2023YFA1010200, the NSFC under Grants No. 12171452 and No. 12231014 and Innovation Program for Quantum Science and Technology 2021ZD0302902.
}
\and
Hong Liu\thanks{Extremal Combinatorics and Probability Group (ECOPRO), Institute for Basic Science (IBS), Daejeon, South Korea. Emails: \texttt{hongliu@ibs.re.kr}. Supported by the Institute for Basic Science (IBS-R029-C4).}
}
\date{}
\begin{document}

\maketitle

\begin{abstract}
We survey recent developments on minimum-degree thresholds in extremal graph theory, with an emphasis on structural refinements of the chromatic-threshold problem. The chromatic threshold asks when a linear minimum-degree condition forces bounded chromatic number in an \(H\)-free graph; its complete classification shows a striking finite trichotomy. We discuss several stronger and related threshold notions in which bounded colourability is replaced by bounded \(H\)-free homomorphic images, bounded neighbourhood complexity via VC-dimension, exact blowup structure for maximal \(H\)-free graphs, and improved quantitative bounds in graph removal lemmas. In contrast with the rigidity of the chromatic threshold, these newer parameters exhibit richer behaviour, including separations from \(\delta_\chi(H)\), new exact values, accumulation points in fixed-chromatic spectra, positivity phenomena for 3-chromatic graphs, asymmetric variants, and failures of monotonicity.
\end{abstract}

\section{Introduction}\label{sec:intro}

A guiding goal in extremal graph theory is to understand what global structure is forced when a large graph avoids a fixed configuration.
Classical results such as Tur\'an's theorem and the Erd\H{o}s--Stone--Simonovits theorem~\cite{1966ES,1946ErodsBAMS}
determine the asymptotic maximum edge density of an $H$-free graph:
if $\chi(H)=r+1$, then
\[
    \ex(n,H)=\Bigl(1-\frac1r+o(1)\Bigr)\binom n2.
\]
Moreover, stability theory shows that near-extremal $H$-free graphs are close to balanced $r$-partite templates.
However, this edge-density viewpoint alone does not control chromatic or bounded-template structure.
Indeed, by Erd\H{o}s' high-girth high-chromatic construction~\cite{1959Erdos}, one can force arbitrarily large chromatic number while avoiding any fixed non-bipartite graph.
Even more, if density is measured only by average degree or by the total number of edges, one may plant such a high-chromatic graph on a small exceptional set inside an otherwise extremal template, changing the edge density by only $o(n^2)$.

This motivates a more robust local density condition: a linear minimum-degree assumption.
A typical threshold problem asks for the least $\alpha$ such that every sufficiently large $H$-free graph $G$ with
$\delta(G)\ge \alpha |V(G)|$
must satisfy a prescribed structural conclusion.
A classical prototype is the theorem of Andr\'asfai, Erd\H{o}s and S\'os~\cite{AES74}: every triangle-free graph $G$ on $n$ vertices with
$\delta(G)> 2n/5$
is bipartite.
More generally, their theorem gives sharp minimum-degree conditions forcing $K_{r+1}$-free graphs to be $r$-partite.
This suggests a softer but more flexible question: for a fixed forbidden graph $H$, how large must the minimum degree be in order to force bounded chromatic number?

This leads to the \emph{chromatic threshold}, introduced implicitly by Erd\H{o}s and Simonovits~\cite{1973ErdosSimonovits}:
\begin{align*}
    \delta_{\chi}(H):=\inf\Bigl\{&\alpha\ge 0:\exists\,C=C(\alpha,H)\ \text{s.t. $\forall$ $n$-vertex $H$-free graph $G$, }\ \\
    &\delta(G)\ge \alpha n \Rightarrow \chi(G)\le C\Bigr\}.
\end{align*}
The study of this parameter began with the triangle case, where Hajnal's construction and Thomassen's matching upper bound give $\delta_\chi(K_3)=1/3$.
It was then extended to cliques by Goddard--Lyle~\cite{goddard_dense_2011} and independently Nikiforov~\cite{nikiforov_chromatic_2010}, while Thomassen showed that longer odd cycles have vanishing chromatic threshold.
The general problem was completely resolved by Allen, B\"ottcher, Griffiths, Kohayakawa and Morris~\cite{ALLEN2013261}.
They proved that if $\chi(H)=r\ge 3$, then
\[
    \delta_\chi(H)\in
    \left\{
        \frac{r-2}{r-1},\,
        \frac{2r-5}{2r-3},\,
        \frac{r-3}{r-2}
    \right\},
\]
with the exact value determined by structural properties of $H$.
Thus, once the chromatic number of $H$ is fixed, the chromatic threshold has a remarkably rigid finite spectrum.

The main purpose of this survey is to discuss what happens when bounded colourability is replaced by stronger or different notions of low complexity.
The chromatic-threshold theorem shows that bounded colourability is governed by a remarkably rigid trichotomy.
A central theme of this survey is that this rigidity is exceptional: when one asks instead for bounded homomorphic templates, bounded neighbourhood complexity, exact blowup structure, or improved removal behaviour, the corresponding minimum-degree thresholds display a much richer and less monotone landscape.

We now turn to several descendants of the chromatic-threshold problem.
The first is the \emph{homomorphism threshold}.
Already in Thomassen's early work on dense triangle-free graphs~\cite{thomassen_chromatic_2002}, one sees, at least implicitly, a strengthening of the chromatic-threshold question: instead of asking only for bounded chromatic number, equivalently for a homomorphism to a bounded complete graph, one asks whether every sufficiently dense \(H\)-free graph admits a homomorphism to a bounded \(H\)-free graph.
Thus the conclusion is no longer merely bounded colourability, but bounded \(H\)-free template structure.

A second group of questions is motivated by VC-dimension.
Originating in the theory of Vapnik--Chervonenkis as a measure of the complexity of set systems, VC-dimension has become a central tool in combinatorics, discrete geometry and learning theory.
In the chromatic-threshold problem it appears naturally through the neighbourhood set system of a graph.
The work of {\L}uczak--Thomass\'e~\cite{luczak_coloring_2010} showed that VC-dimension can play an important role both as a proof tool and as a guiding structural principle.
Related neighbourhood-complexity ideas also play a role in the final classification theorem of Allen, B\"ottcher, Griffiths, Kohayakawa and Morris~\cite{ALLEN2013261}.
This role of VC-dimension led to minimum-degree thresholds for bounded VC-dimension itself, and also to VC-bounded variants of chromatic and homomorphism thresholds; the precise references and formulations will be given in \Cref{sec:vc-thresholds}.

A further strengthening of the homomorphism viewpoint is the blowup threshold.
Introduced in the context of bounded-template problems~\cite{huang_interpolating_2025}, this parameter asks not only whether a dense \(H\)-free graph maps to a bounded \(H\)-free template, but whether, under a maximality assumption, the graph itself must be a blowup of such a bounded template.

It is useful to keep the following informal hierarchy in mind:
\begin{align*}
\text{bounded blowup structure}
&\Longrightarrow \text{bounded }H\text{-free homomorphic image}\\
&\Longrightarrow \text{bounded colouring}.
\end{align*}
These implications compare the structural conclusions forced by the
parameters; the associated thresholds need not be equal. The universal
threshold inequalities that follow from them are recorded in
Section~\ref{sec:conclusion}.

Finally, the same minimum-degree philosophy also interacts with graph removal.
The ordinary graph removal lemma has very poor quantitative dependence in general.
Minimum-degree removal thresholds ask whether a linear lower bound on \(\delta(G)\) can improve this dependence, for instance to a polynomial or even linear bound.
Thus minimum degree is used not to force bounded colourability or bounded templates, but to improve quantitative removal and supersaturation phenomena.

We survey these derived minimum-degree threshold notions, summarize their recent progress, and compare their behaviour with the original chromatic threshold.
In some cases they behave like \(\delta_\chi(H)\), while in others they reveal new values, richer spectra, positivity phenomena, or failures of monotonicity that have no analogue in the classical chromatic-threshold classification.

\medskip
\noindent\textbf{Outline.}
In \Cref{sec:hom-threshold} we survey homomorphism thresholds, beginning with the triangle and clique cases and then discussing odd cycles, forbidden families of short odd cycles, and recent phenomena such as new values, accumulation points, and non-monotonicity.
In \Cref{sec:vc-thresholds} we discuss VC thresholds and VC-bounded threshold problems, including VC-bounded chromatic and homomorphism thresholds and asymmetric variants interpolating between them.
In \Cref{sec:blowup-threshold} we introduce blowup thresholds and compare them with chromatic and homomorphism thresholds.
In \Cref{sec:removal} we briefly discuss minimum-degree removal thresholds.
In \Cref{sec:conclusion} we collect several connections and possible organising principles for this broader minimum-degree threshold landscape.

\section{Homomorphism thresholds}\label{sec:hom-threshold}

We now formalize the bounded-template viewpoint mentioned in the introduction.
The chromatic threshold asks when a linear minimum-degree condition forces
bounded chromatic number. Since a proper \(q\)-colouring of a graph \(G\) is
equivalently a homomorphism from \(G\) to the complete graph \(K_q\), bounded
colourability can be rephrased as the existence of a homomorphism to some
bounded graph. The homomorphism-threshold problem asks for a stronger conclusion:
can the bounded target be chosen to be \(H\)-free?

Given graphs \(G\) and \(F\), we write
$G \xrightarrow{\mathrm{hom}} F$
if there exists a map \(\varphi:V(G)\to V(F)\) such that
\[
    uv\in E(G)\quad\Longrightarrow\quad
    \varphi(u)\varphi(v)\in E(F).
\]
Thus a proper \(q\)-colouring of \(G\) is precisely a homomorphism
\(G\xrightarrow{\mathrm{hom}}K_q\). In the homomorphism-threshold setting,
however, the target graph is required not only to be bounded, but also to be
\(H\)-free. This additional requirement is the essential difference between
bounded colourability and bounded \(H\)-free template structure.

\begin{definition}[Homomorphism threshold]\label{def:hom-threshold}
For a fixed graph \(H\), the \emph{homomorphism threshold} of \(H\) is
\begin{align*}
    \delta_{\mathrm{hom}}(H)
    :=
    \inf\Bigl\{
    &\alpha> 0:
    \exists\, H\text{-free }F=F(\alpha,H)\ \text{s.t. } \forall\,
    n\text{-vertex }H\text{-free graph }G,\ \\
    &\delta(G)\ge \alpha n \Rightarrow G\xrightarrow{\mathrm{hom}}F
    \Bigr\}.
\end{align*}
\end{definition}

Equivalently, \(\delta_{\mathrm{hom}}(H)\) is the minimum-degree density above
which every \(H\)-free graph admits a homomorphism to an \(H\)-free graph with
boundedly many vertices. Since every homomorphism to a bounded graph gives a
bounded colouring, we always have
\[
    \delta_\chi(H)\le \delta_{\mathrm{hom}}(H).
\]
The inequality expresses the basic hierarchy between the two problems:
the chromatic threshold asks for bounded complexity, whereas the homomorphism
threshold asks for bounded complexity of a form that still avoids \(H\).

The first basic test case is when the forbidden graph is a clique.
In this case the stronger requirement that the bounded target also be
\(K_s\)-free turns out not to change the threshold.
This makes cliques the natural starting point for the theory.

\subsection{Triangles and cliques}\label{subsec:hom-triangles-cliques}

We begin with triangle-free graphs. H\"aggkvist~\cite{haggkvist_odd_1982}
proved an early bounded-template result: every triangle-free graph \(G\) on
\(n\) vertices with
$\delta(G)>3n/8$
is homomorphic to \(C_5\). 
This gives the first explicit upper bound
$\delta_{\mathrm{hom}}(K_3)\le 3/8$.
The sharp threshold was first obtained by {\L}uczak~\cite{luczak_structure_2006}.

\begin{theorem}[{\L}uczak]\label{thm:hom-triangle}
One has
\[
    \delta_{\mathrm{hom}}(K_3)\le \frac13.
\]
Consequently,
\[
    \delta_{\mathrm{hom}}(K_3)
    =
    \delta_\chi(K_3)
    =
    \frac13.
\]
\end{theorem}

The lower bound follows immediately from the chromatic-threshold lower bound
\(\delta_\chi(K_3)\ge 1/3\), since
$\delta_\chi(K_3)\le \delta_{\mathrm{hom}}(K_3)$.

\subsubsection{Proof of \Cref{thm:hom-triangle} via VC-dimension theory}
\label{subsubsec:triangle-vc-packing}

We give a short proof of \Cref{thm:hom-triangle}
using the VC-dimension and packing-lemma method behind the blowup theorem of
Liu--Shangguan--Skokan--Xu~\cite{liu_beyond_2024}. The point of the argument is
not merely to obtain a bounded colouring, but to produce a bounded
triangle-free quotient.

We first recall the relevant VC-dimension terminology.

\begin{definition}[VC-dimension]\label{def:VC-dim-set-system}
Let \(X\) be a finite set and let \(\mathcal F\subseteq 2^X\) be a set system.
A subset \(S\subseteq X\) is \emph{shattered} by \(\mathcal F\) if
\[
    \{F\cap S:F\in\mathcal F\}=2^S.
\]
The \emph{VC-dimension} of \(\mathcal F\), denoted
\(\mathrm{VC}(\mathcal F)\), is the maximum size of a subset of \(X\) shattered
by \(\mathcal F\).
\end{definition}

For a graph \(G\), we will use the VC-dimension of its neighbourhood set system
\[
    \mathcal N(G):=\{N_G(v):v\in V(G)\}\subseteq 2^{V(G)}.
\]
Thus $\mathrm{VC}(G):=\mathrm{VC}(\mathcal N(G))$.
This is the same notion of graph VC-dimension that appears in the
VC-dimension approach to chromatic thresholds, for instance in the work of
{\L}uczak--Thomass\'e~\cite{luczak_coloring_2010}.

The only general VC-theoretic input we need is Haussler's packing lemma.

\begin{lemma}[Haussler packing lemma, graph form~\cite{haussler_packing_1995,liu_beyond_2024}]
\label{lem:haussler-packing-graph}
Let \(G\) be an \(n\)-vertex graph with $\mathrm{VC}(G)\le d$.
For every \(s>0\), there is a partition
\begin{align*}
    V(G)=V_1\cup\cdots\cup V_m
    \quad\text{  with  } \quad
    m\le e(d+1)\left(\frac{2en}{s}\right)^d
\end{align*}
such that for every \(i\in[m]\) and every \(u,v\in V_i\),
$|N_G(u)\triangle N_G(v)|\le 2s$.
\end{lemma}

We will also use the following result without proof.

\begin{lemma}[VC-dimension reduction above \(1/3\), maximal form~\cite{luczak_coloring_2010}]
\label{lem:vc-reduction-triangle}
For every \(\varepsilon>0\), there exists \(d=d(\varepsilon)\) such that every maximal triangle-free graph \(G\) on \(n\) vertices with
$\delta(G)\ge \left(\frac13+\varepsilon\right)n$
satisfies $\mathrm{VC}(G)\le d$.
\end{lemma}

The maximality assumption is essential here: dense bipartite graphs can be triangle-free, have minimum degree larger than \(n/3\), and still have unbounded neighbourhood VC-dimension.

\begin{proof}[Proof of \Cref{thm:hom-triangle}]
It remains to prove the upper bound
\[
    \delta_{\mathrm{hom}}(K_3)\le \frac13.
\]
Equivalently, we show that for every \(\varepsilon>0\), every triangle-free
graph \(G\) with
\[
    \delta(G)\ge \left(\frac13+\varepsilon\right)|V(G)|
\]
admits a homomorphism to a triangle-free graph whose order is bounded in terms
of \(\varepsilon\) only.

We may assume that \(G\) is maximal triangle-free. Indeed, starting from \(G\),
add edges until obtaining a maximal triangle-free supergraph \(G'\) on the same
vertex set. The minimum degree does not decrease, and every homomorphism
\(G'\to F\) to a triangle-free graph is also a homomorphism \(G\to F\).
Let \(n=|V(G)|\), and assume from now on that \(G\) is maximal triangle-free and
\[
    \delta(G)\ge \left(\frac13+\varepsilon\right)n.
\]
We claim the following.

\begin{claim}\label{claim:triangle-large-codegree}
For every non-edge \(uv\notin E(G)\),
$|N_G(u)\cap N_G(v)|\ge 3\varepsilon n$.
\end{claim}

Assuming Claim~\ref{claim:triangle-large-codegree}, we finish the proof.
Set $\eta:=3\varepsilon$.
Thus every non-edge \(uv\notin E(G)\) satisfies
\begin{align}\label{ineq:hom-triangle-large-codegree}
    |N_G(u)\cap N_G(v)|\ge \eta n.
\end{align}

By Lemma~\ref{lem:vc-reduction-triangle}, there exists \(d=d(\varepsilon)\) such
that $\mathrm{VC}(G)\le d$.
Apply Lemma~\ref{lem:haussler-packing-graph} with
$s:=\eta n/10$.

We obtain a partition
$V(G)=V_1\cup\cdots\cup V_m$
where $m\le e(d+1)\left(20e/\eta\right)^d$,
and for every \(i\in[m]\) and every \(u,v\in V_i\),
\begin{align}\label{ineq:hom-triangle-sym-difference}
    |N_G(u)\triangle N_G(v)|\le 2s.
\end{align}
In particular, \(m\) is bounded in terms of \(\varepsilon\) only.

We next show that this partition is essentially a blowup partition. 
Firstly, for each \(i\in [m]\), \(G[V_i]\) is indeed an independent set since \(G\) is triangle-free and any two vertices in \(V_i\) have at least one common neighbour.
Indeed, if \(u,v\in V_i\), then by~\eqref{ineq:hom-triangle-sym-difference}, we have \(|N_G(u)\triangle N_G(v)|\le \eta n/5\), and hence
$|N_G(u)\cap N_G(v)|\ge d(u)-|N_G(u)\setminus N_G(v)|>0$.
Thus \(u\) and \(v\) have a common neighbour, and so \(uv\notin E(G)\) since \(G\) is triangle-free.

It remains to show that, for any \(i\neq j\) in \([m]\), if there exists \((u,w)\in V_i\times V_j\) such that \(uw\) is a non-edge in \(G\), then the bipartite graph \(G[V_i,V_j]\) is empty.
Otherwise, assume that there exists \((u',w')\in V_i\times V_j\) such that \(u'w'\in E(G)\).
By~\eqref{ineq:hom-triangle-large-codegree}, we know that \(|N_G(u)\cap N_G(w)|\ge \eta n\).
By~\eqref{ineq:hom-triangle-sym-difference}, we know that
\[
    |N_G(u)\triangle N_G(u')|,\ |N_G(w)\triangle N_G(w')|\le \frac{\eta n}{5}.
\]
Therefore,
\[
    |N_G(u')\cap N_G(w')|
    \ge |N_G(u)\cap N_G(w)|-\frac{\eta n}{5}-\frac{\eta n}{5}
    \ge \frac{3\eta n}{5}>0.
\]
Thus \(u'\) and \(w'\) have a common neighbour, which together with the edge \(u'w'\) forms a triangle, a contradiction.
Hence \(G[V_i,V_j]\) is empty whenever it contains a non-edge. Equivalently, for every distinct \(i,j\in[m]\), the bipartite graph \(G[V_i,V_j]\) is either complete or empty.

Let \(F_\varepsilon\) be the quotient graph of this partition: its vertices are the parts \(V_1,\dots,V_m\), and two parts are adjacent in \(F_\varepsilon\) if and only if they are complete to each other in \(G\).
Since each \(V_i\) is independent and every pair of parts is either complete or empty, the natural map sending each vertex of \(G\) to the part containing it is a homomorphism \(G\xrightarrow{\mathrm{hom}}F_\varepsilon\).
Moreover, \(F_\varepsilon\) is triangle-free, because any triangle in \(F_\varepsilon\) would lift to a triangle in \(G\).
Finally, \(|V(F_\varepsilon)|=m\le e(d+1)(20e/\eta)^d\), which is bounded in terms of \(\varepsilon\) only.
This proves the desired upper bound.

It remains to prove Claim~\ref{claim:triangle-large-codegree}. Let \(uv\notin E(G)\).
Since \(G\) is maximal triangle-free, the vertices \(u\) and \(v\) have a common neighbour; choose \(w\in N_G(u)\cap N_G(v)\).
Since \(G\) is triangle-free, \(N_G(w)\) is disjoint from \(N_G(u)\cup N_G(v)\). Hence
\[
    n
    \ge |N_G(u)\cup N_G(v)|+|N_G(w)|
    =
    d(u)+d(v)-|N_G(u)\cap N_G(v)|+d(w).
\]
Rearranging gives \(|N_G(u)\cap N_G(v)|\ge d(u)+d(v)+d(w)-n\).
Using \(\delta(G)\ge (1/3+\varepsilon)n\), we obtain \(|N_G(u)\cap N_G(v)|\ge 3(1/3+\varepsilon)n-n=3\varepsilon n\).
This proves Claim~\ref{claim:triangle-large-codegree}, and hence completes the proof of Theorem~\ref{thm:hom-triangle}.
\end{proof}

\subsubsection{General cliques}

The triangle case extends to all cliques.

\begin{theorem}[Goddard--Lyle]\label{thm:hom-cliques}
For every \(s\ge 3\),
\[
    \delta_{\mathrm{hom}}(K_s)
    =
    \delta_\chi(K_s)
    =
    \frac{2s-5}{2s-3}.
\]
\end{theorem}

The equality for the chromatic threshold of cliques was proved by
Goddard--Lyle and independently by Nikiforov~\cite{goddard_dense_2011,nikiforov_chromatic_2010}.
The stronger homomorphism-threshold statement is due to Goddard--Lyle:
for every \(\varepsilon>0\), every sufficiently large \(K_s\)-free graph \(G\)
with
\[
    \delta(G)\ge
    \left(\frac{2s-5}{2s-3}+\varepsilon\right)|V(G)|
\]
admits a homomorphism to a bounded \(K_s\)-free graph.

\begin{sloppypar}
Thus cliques exhibit the most optimistic behaviour one could hope for:
the same minimum-degree threshold that forces bounded colourability also forces
a bounded \(K_s\)-free template. The original proof uses Szemer\'edi's
regularity lemma and therefore gives very large templates.
Oberkampf and Schacht later gave a more efficient proof, with better quantitative
bounds on the size of the template~\cite{oberkampf_structure_2020}.
Liu--Shangguan--Skokan--Xu further strengthened this picture by proving a sharp
blowup phenomenon, using the neighbourhood-complexity and packing ideas
illustrated above in the triangle case~\cite{liu_beyond_2024}.
\end{sloppypar}

\subsection{Odd cycles and the first separation}
\label{subsec:hom-odd-cycles}

The clique case might suggest that the homomorphism threshold should coincide
with the chromatic threshold. Odd cycles show that this is false.

For longer odd cycles, the chromatic threshold vanishes. Thomassen~\cite{2007OddCycleChromatic}
proved that for every \(k\ge 3\),
$\delta_\chi(C_{2k-1})=0$.
Equivalently, for every fixed positive minimum-degree proportion, forbidding
one fixed odd cycle \(C_{2k-1}\) already forces bounded chromatic number.

The homomorphism threshold behaves differently. Ebsen--Schacht~\cite{2020COMBHomoOddCycle}
proved the general upper bound
$\delta_{\mathrm{hom}}(C_{2k-1})\le \frac{1}{2k-1}$ for each $k\ge 3$.
Using topological methods, Sankar~\cite{sankar2022homotopy} proved the first
non-trivial lower bound for a single odd cycle.

\begin{theorem}[Ebsen--Schacht~\cite{2020COMBHomoOddCycle}; Sankar~\cite{sankar2022homotopy}]
\label{thm:odd-cycle-hom-bounds}
For every \(k\ge 3\),
\[
    k^{-(8+o(1))k^2}
    \le
    \delta_{\mathrm{hom}}(C_{2k-1})
    \le
    \frac{1}{2k-1}.
\]
In particular, $\delta_{\mathrm{hom}}(C_{2k-1})>0$.
\end{theorem}

Sankar's proof introduces a topological obstruction to bounded
\(C_{2k-1}\)-free homomorphic images. The construction starts from finite
approximations to suitable Borsuk graphs, which are dense, \(C_{2k-1}\)-free,
and have large chromatic number. The main new ingredient is a graph-theoretic
notion of homotopy for walks. This allows one to control how odd closed walks
can be identified under a homomorphism. Roughly speaking, if such a Borsuk
approximation admitted a homomorphism to a bounded \(C_{2k-1}\)-free graph, then
some short odd-cycle information would have to collapse in a way forbidden by
the discrete homotopy invariants. This contradiction rules out bounded
\(C_{2k-1}\)-free homomorphic images. The framework is closely related to the
fundamental group of the Lov\'asz neighbourhood complex, but the discrete
homotopy language is better suited for tracking odd walks.

Consequently, for every \(k\ge 3\),
$\delta_\chi(C_{2k-1})
    =
    0
    <
    \delta_{\mathrm{hom}}(C_{2k-1})$.
Thus odd cycles give the first examples of single forbidden graphs \(H\) for
which the homomorphism threshold is strictly larger than the chromatic threshold.
Despite this separation, the exact value of
\(\delta_{\mathrm{hom}}(C_{2k-1})\) remains unknown. The current upper and lower
bounds are still far apart, and even the natural possibility that the
Ebsen--Schacht upper bound is sharp is open.

\begin{problem}\label{prob:hom-odd-cycle-exact}
Determine \(\delta_{\mathrm{hom}}(C_{2k-1})\) for \(k\ge 3\). In particular, is
\[
    \delta_{\mathrm{hom}}(C_{2k-1})
    =
    \frac{1}{2k-1}
\]
for every \(k\ge 3\)?
\end{problem}

\subsection{More separations between
\texorpdfstring{$\delta_{\chi}$}{delta-chi} and
\texorpdfstring{$\delta_{\mathrm{hom}}$}{delta-hom}}
\label{subsec:hom-more-separations}

Odd cycles show that the inequality
$\delta_\chi(H)\le \delta_{\mathrm{hom}}(H)$
can be strict for a single forbidden graph. Recent work of Gishboliner and the
authors~\cite{gishboliner_huang_liu_vc_thresholds} shows that this is not an isolated phenomenon. In fact, homomorphism
thresholds have a much richer spectrum than chromatic thresholds.

It is useful to make this spectral viewpoint explicit. For a threshold parameter
\(\delta_*\), where \(*\in\{\chi,\mathrm{hom}\}\), define
\[
    \Delta_*
    :=
    \{\delta_*(H): \chi(H)\ge 3\},
\]
and, for each integer \(r\ge 3\),
\[
    \Delta_*^{(r)}
    :=
    \{\delta_*(H): \chi(H)=r\}.
\]
Thus \(\Delta_*\) is the non-bipartite spectrum of the threshold \(\delta_*\),
and \(\Delta_*^{(r)}\) is its spectrum restricted to \(r\)-chromatic forbidden
graphs.

The corresponding chromatic-threshold spectra are extremely rigid: by the
classification theorem of Allen, B\"ottcher, Griffiths, Kohayakawa and
Morris~\cite{ALLEN2013261}, for every \(r\ge 3\),
\[
    \Delta_{\chi}^{(r)}
    =
    \left\{
        \frac{r-3}{r-2},
        \frac{2r-5}{2r-3},
        \frac{r-2}{r-1}
    \right\}.
\]
The next results show that \(\Delta_{\mathrm{hom}}\) behaves very differently.

We first consider complete tripartite graphs of the form \(K_{1,s,t}\). These
already produce infinitely many new values and many new examples separating
\(\delta_{\mathrm{hom}}\) from \(\delta_\chi\).

\begin{theorem}[Gishboliner--Huang--Liu~\cite{gishboliner_huang_liu_vc_thresholds}]
\label{thm:hom-K1st}
Let \(1\le s\le t\). Then
\[
    \delta_{\mathrm{hom}}(K_{1,s,t})
    \ge
    \max\left\{
        \frac13,\,
        \frac{s}{1+s+t}
    \right\}.
\]
Moreover, if
\[
    t-s\le \frac{\sqrt{s}}2
    \qquad\text{and}\qquad
    3s\ge 2+2t,
\]
then
\[
    \delta_{\mathrm{hom}}(K_{1,s,t})
    =
    \frac{s}{1+s+t}.
\]
\end{theorem}

Since \(K_{1,s,t}\) is 3-chromatic and $\delta_\chi(K_{1,s,t})=1/3$, 
Theorem~\ref{thm:hom-K1st} gives many explicit separations between chromatic and
homomorphism thresholds. For example, when \(t\) is close to \(s\), the value
\(s/(1+s+t)\) is strictly larger than \(1/3\). Thus the homomorphism threshold
can distinguish between complete tripartite forbidden graphs which have the
same chromatic threshold.

The next result shows that this extra complexity is not confined to the
3-chromatic case.

\begin{theorem}[Gishboliner--Huang--Liu~\cite{gishboliner_huang_liu_vc_thresholds}]
\label{thm:hom-accumulation}
For every integer \(r\ge 3\), the set \(\Delta_{\mathrm{hom}}^{(r)}\) has an
accumulation point. More precisely,
$\frac{r-2}{r-1}$
is an accumulation point of \(\Delta_{\mathrm{hom}}^{(r)}\).
\end{theorem}

A crucial simplification in the classification of chromatic thresholds by
Allen--B\"ottcher--Griffiths--Kohayakawa--Morris~\cite{ALLEN2013261} is
monotonicity. Namely, if \(H_1\subseteq H_2\), then every \(H_1\)-free graph is
\(H_2\)-free, and hence
\[
    \delta_\chi(H_1)\le \delta_\chi(H_2).
\]
This monotonicity allows one to reduce the classification to suitable minimal or maximal
forbidden configurations. The complete tripartite examples above show that this
basic principle fails for the homomorphism threshold.

\begin{cor}\label{cor:hom-nonmonotone}
The homomorphism threshold is not monotone under taking subgraphs or induced
subgraphs.
\end{cor}

\begin{proof}
Let \(s\ge 4\). Applying Theorem~\ref{thm:hom-K1st} to \(K_{1,s,s}\) and
\(K_{1,s,s+1}\), we obtain
\[
    \delta_{\mathrm{hom}}(K_{1,s,s})
    =
    \frac{s}{2s+1}
    >
    \frac{s}{2s+2}
    =
    \delta_{\mathrm{hom}}(K_{1,s,s+1}).
\]
However, \(K_{1,s,s}\) is an induced subgraph of \(K_{1,s,s+1}\). Thus passing
to a supergraph of the forbidden graph can decrease the homomorphism threshold.
Hence \(\delta_{\mathrm{hom}}\) is not monotone under taking induced subgraphs,
and therefore is not monotone under taking subgraphs either.
\end{proof}

These results show that the homomorphism threshold is more delicate than the
chromatic threshold in several ways. It has new exact values, fixed-chromatic
spectra with accumulation points, and graph-function behaviour which is absent
from the chromatic-threshold theory. Thus determining homomorphism thresholds
appears to require genuinely finer information about the forbidden graph and
about bounded \(H\)-free templates.

\subsection{Homomorphism thresholds for families}
\label{subsec:hom-families}

The definition of the homomorphism threshold extends naturally from a single
forbidden graph to a family of forbidden graphs. If \(\mathcal F\) is a family
of graphs, we say that a graph is \(\mathcal F\)-free if it contains no member
of \(\mathcal F\) as a subgraph. Define
\begin{align*}
    \delta_{\mathrm{hom}}(\mathcal F)
    :=
    \inf\Bigl\{
    &\alpha\ge 0:
    \exists\, \mathcal F\text{-free }\Gamma=\Gamma(\alpha,\mathcal F)\ \text{s.t. } \forall\,
    n\text{-vertex }\mathcal F\text{-free graph }G,\ \\
    &\delta(G)\ge \alpha n \Rightarrow G\xrightarrow{\mathrm{hom}}\Gamma
    \Bigr\}.
\end{align*}
Thus the usual definition is recovered when \(\mathcal F=\{H\}\).

The most important examples are the families of short odd cycles. Let
\[
    \mathcal C_{2\ell-1}:=\{C_3,C_5,\dots,C_{2\ell-1}\}.
\]
Equivalently, an \(\mathcal C_{2\ell-1}\)-free graph has odd girth at least
\(2\ell+1\).

The first nontrivial case was determined by Letzter and Snyder
\cite{letzter_homomorphism_2019}, who proved
\[
    \delta_{\mathrm{hom}}(\mathcal C_5)
    =
    \delta_{\mathrm{hom}}(\{C_3,C_5\})
    =
    \frac15.
\]
Ebsen and Schacht~\cite{2020COMBHomoOddCycle} extended this to all families of short odd cycles.

\begin{theorem}[Letzter--Snyder~\cite{letzter_homomorphism_2019}; Ebsen--Schacht~\cite{2020COMBHomoOddCycle}]
\label{thm:hom-short-odd-cycles}
For every \(\ell\ge 2\),
\[
    \delta_{\mathrm{hom}}(\mathcal C_{2\ell-1})
    =
    \frac{1}{2\ell-1}.
\]
\end{theorem}

The lower-bound examples are generalized Andr\'asfai-type graphs: they have odd
girth at least \(2\ell+1\), minimum-degree ratio asymptotic to
\(1/(2\ell-1)\), and the property that any sufficiently small homomorphic image
necessarily creates a forbidden shorter odd cycle. The upper bound shows that
above this density, every graph of odd girth at least \(2\ell+1\) admits a
homomorphism to a bounded graph of the same odd-girth type.

This theorem should be compared with the single-cycle problem. For the family
\(\mathcal C_{2\ell-1}\), forbidding all shorter odd cycles imposes a genuine
odd-girth condition, and the homomorphism threshold is known exactly. For the
single forbidden graph \(C_{2\ell-1}\), shorter odd cycles are still allowed;
this extra freedom is one reason why the exact value of
\(\delta_{\mathrm{hom}}(C_{2\ell-1})\) remains open.

\subsection{Asymmetric homomorphism thresholds}
\label{subsubsec:asymmetric-hom-thresholds}

There is also a natural asymmetric version of the homomorphism-threshold
problem. 
Given two families of graphs \(\mathcal F_1\) and \(\mathcal F_2\), call a real
number \(\alpha\ge 0\) \emph{admissible} if there exists an
\(\mathcal F_2\)-free graph \(\Gamma=\Gamma(\alpha,\mathcal F_1,\mathcal F_2)\) such that
every \(n\)-vertex \(\mathcal F_1\)-free graph \(G\) with
\(\delta(G)\ge \alpha n\) satisfies \(G\xrightarrow{\mathrm{hom}}\Gamma\). Define
\[
    \delta_{\mathrm{hom}}(\mathcal F_1;\mathcal F_2)
    :=
    \inf\{\alpha\ge 0:\alpha \text{ is admissible}\}.
\]
Thus the usual homomorphism threshold is the diagonal case
\[
    \delta_{\mathrm{hom}}(\mathcal F)
    =
    \delta_{\mathrm{hom}}(\mathcal F;\mathcal F).
\]

This definition separates the assumption placed on the source graph from the
structure required of the target. For families of short odd cycles, this leads
to a striking zero-threshold phenomenon. Recall that
\[
    \mathcal C_{2t+1}:=\{C_3,C_5,\dots,C_{2t+1}\}.
\]

\begin{theorem}[Gishboliner--Hurley--Wigderson~\cite{gishboliner_hurley_wigderson_asymmetric_2025}]
\label{thm:asymmetric-hom-odd-cycles}
For every \(t\ge 1\),
\[
    \delta_{\mathrm{hom}}(\mathcal C_{2t+5};\mathcal C_{2t+1})=0.
\]
\end{theorem}

Let us unpack the first case. When \(t=1\), 
Theorem~\ref{thm:asymmetric-hom-odd-cycles} says that every graph of odd girth
at least \(9\) and linear minimum degree admits a homomorphism to a bounded
triangle-free graph. This is genuinely asymmetric: the source graph is required
to avoid all odd cycles up to length \(7\), while the target is only required to
be triangle-free.

The theorem therefore shows that, once the source graph has slightly larger
odd girth than the target is required to have, the minimum-degree threshold can
collapse to zero. Gishboliner--Hurley--Wigderson conjectured that the gap in
Theorem~\ref{thm:asymmetric-hom-odd-cycles} can be reduced by two.

\begin{conj}[Gishboliner--Hurley--Wigderson~\cite{gishboliner_hurley_wigderson_asymmetric_2025}]
\label{conj:asymmetric-hom-odd-cycles}
For every \(t\ge 1\),
\[
    \delta_{\mathrm{hom}}(\mathcal C_{2t+3};\mathcal C_{2t+1})=0.
\]
\end{conj}

\begin{remark}
Conjecture~\ref{conj:asymmetric-hom-odd-cycles} would be best possible in the
following sense. On the diagonal, the theorem of Ebsen--Schacht gives
\[
    \delta_{\mathrm{hom}}(\mathcal C_{2t+1};\mathcal C_{2t+1})
    =
    \frac{1}{2t+1}
    >0.
\]
Thus one cannot replace \(\mathcal C_{2t+3}\) in the source by
\(\mathcal C_{2t+1}\). 
\end{remark}

Overall, homomorphism thresholds refine chromatic thresholds by requiring the
bounded target to retain the forbidden-graph constraint. The results in this
section show that this extra requirement already leads to new values,
separations, accumulation phenomena, and non-monotonicity. In
Section~\ref{sec:blowup-threshold}, we will return to this bounded-template
viewpoint through the stronger notion of blowup thresholds.


\section{Thresholds via VC-dimension}
\label{sec:vc-thresholds}

VC-dimension has already appeared several times as a proof tool in the theory
of minimum-degree thresholds. In the chromatic-threshold problem, it provides an
alternative to regularity-based arguments for converting dense local information
into bounded colourability. In the homomorphism-threshold problem, VC-type
partition and packing ideas help produce bounded homomorphic images and, in
stronger forms, bounded blowup quotients.

The first indication that VC-dimension can affect minimum-degree thresholds
comes already from triangle-free graphs. The classical chromatic threshold is $\delta_\chi(K_3)=1/3$.
However, the VC-dimension method of {\L}uczak--Thomass\'e \cite{luczak_coloring_2010}
shows that, if one restricts to triangle-free graphs of bounded VC-dimension,
then every positive linear minimum-degree condition already forces bounded
chromatic number. In other words, after imposing bounded VC-dimension, the
minimum-degree threshold for bounded colourability in the triangle-free setting
drops from \(1/3\) to \(0\).
This phenomenon suggests that, in the triangle case, VC-dimension is the bottleneck behind the classical chromatic threshold.

This example suggests a natural question: when is VC-dimension a genuine
bottleneck in a minimum-degree threshold problem? 

\subsection{Chromatic and homomorphism thresholds under bounded VC-dimension}
\label{subsec:vc-bounded-chromatic-hom}

The triangle example motivates studying how minimum-degree thresholds change
under the additional assumption of bounded VC-dimension. 
This was
formalised by the second author, Shangguan, Skokan and Xu~\cite{liu_beyond_2024},
who introduced VC-bounded versions of the chromatic and homomorphism thresholds.

\begin{definition}[VC-bounded chromatic and homomorphism thresholds]
\label{def:vc-bounded-thresholds}
Let \(H\) be a fixed graph.
Call a real number \(\alpha\ge 0\) \emph{VC-chromatically admissible} for \(H\)
if, for every \(d\in\mathbb N\), there exists a constant
\(C=C(\alpha,H,d)\) such that every \(n\)-vertex \(H\)-free graph \(G\) with $\mathrm{VC}(G)\le d$ and $\delta(G)\ge \alpha n$ satisfies \(\chi(G)\le C\). Define
\[
    \delta_\chi^{\mathrm{VC}}(H)
    :=
    \inf\{\alpha\ge 0:\alpha \text{ is VC-chromatically admissible for } H\}.
\]

Similarly, call a real number \(\alpha\ge 0\) \emph{VC-homomorphically
admissible} for \(H\) if, for every \(d\in\mathbb N\), there exists an
\(H\)-free graph \(F=F(\alpha,H,d)\) such that every \(n\)-vertex \(H\)-free
graph \(G\) with
$\mathrm{VC}(G)\le d$ and $\delta(G)\ge \alpha n$
satisfies $G\xrightarrow{\mathrm{hom}}F$.
Define
\[
    \delta_{\mathrm{hom}}^{\mathrm{VC}}(H)
    :=
    \inf\{\alpha\ge 0:\alpha \text{ is VC-homomorphically admissible for } H\}.
\]
\end{definition}

Since we restrict to the smaller class of \(H\)-free graphs with bounded
VC-dimension, we trivially have
\[
    \delta_\chi^{\mathrm{VC}}(H)\le \delta_\chi(H)
    \qquad\text{and}\qquad
    \delta_{\mathrm{hom}}^{\mathrm{VC}}(H)\le \delta_{\mathrm{hom}}(H).
\]

For cliques, these two VC-bounded thresholds behave quite differently. The
chromatic threshold drops under the bounded-VC assumption, while the
homomorphism threshold does not.

\begin{theorem}[Kim--Liu--Shangguan--Wang--Wu--Xue~\cite{kim_stability_2025};
Liu--Shangguan--Skokan--Xu~\cite{liu_beyond_2024}]
\label{thm:vc-bounded-cliques}
For every \(s\ge 3\),
\[
    \delta_\chi^{\mathrm{VC}}(K_s)
    =
    \frac{s-3}{s-2}
    <
    \delta_{\chi}(K_s)
    =
    \frac{2s-5}{2s-3},
\]
whereas
\[
    \delta_{\mathrm{hom}}^{\mathrm{VC}}(K_s)
    =
    \delta_{\mathrm{hom}}(K_s)
    =
    \frac{2s-5}{2s-3}.
\]
\end{theorem}

Thus bounded VC-dimension removes the obstruction responsible for the classical
chromatic threshold of cliques, but it does not remove the obstruction for
bounded \(K_s\)-free homomorphic images. 
In other words, VC-dimension is a
bottleneck for the chromatic clique threshold, but not for the homomorphism clique threshold.

The clique result is part of a more general picture. Gishboliner and the
authors~\cite{gishboliner_huang_liu_vc_thresholds} determined the
VC-bounded chromatic threshold for every graph.

\begin{definition}[Decomposition family]
\label{def:decomposition-family}
Let \(H\) be a graph with \(\chi(H)=r\ge 3\). The decomposition family
\(\mathcal M(H)\) consists of all bipartite graphs obtained by deleting
\(r-2\) colour classes from \(H\) under a proper \(r\)-colouring of \(H\);
equivalently, these are the subgraphs induced by the union of any two
colour classes in a proper \(r\)-colouring.
\end{definition}

\begin{theorem}[Gishboliner--Huang--Liu~\cite{gishboliner_huang_liu_vc_thresholds}]
\label{thm:vc-bounded-chromatic-general}
Let \(H\) be a graph with \(\chi(H)=r\ge 3\). Then
\[
    \delta_\chi^{\mathrm{VC}}(H)
    \in
    \left\{
        \frac{r-2}{r-1},
        \frac{r-3}{r-2}
    \right\}.
\]
More precisely,
\[
    \delta_\chi^{\mathrm{VC}}(H)=\frac{r-3}{r-2}
\]
if and only if the decomposition family \(\mathcal M(H)\) contains a forest;
otherwise,
\[
    \delta_\chi^{\mathrm{VC}}(H)=\frac{r-2}{r-1}.
\]
\end{theorem}

This should be compared with the classical chromatic-threshold classification
of Allen--B\"ottcher--Griffiths--Kohayakawa--Morris. If \(\chi(H)=r\ge 3\),
then \(\delta_\chi(H)\) can take one of the three values
\[
    \frac{r-2}{r-1},
    \qquad
    \frac{2r-5}{2r-3},
    \qquad
    \frac{r-3}{r-2}.
\]
Theorem~\ref{thm:vc-bounded-chromatic-general} says that, under bounded
VC-dimension, the intermediate value
\[
    \frac{2r-5}{2r-3}
\]
disappears. Thus bounded VC-dimension changes the chromatic threshold exactly
in the middle, or clique-type, regime of the classical classification. In this
sense, the intermediate chromatic threshold is caused by high neighbourhood
complexity: once this complexity is ruled out by a bounded-VC assumption, the
threshold drops to the near-acyclic value \((r-3)/(r-2)\).

By contrast, the corresponding picture for homomorphism thresholds remains much
less understood. The clique case in Theorem~\ref{thm:vc-bounded-cliques} shows
that bounded VC-dimension does not always lower the homomorphism threshold.
However, it is not clear whether this is a special feature of cliques or a more
general phenomenon.

\begin{problem}\label{prob:vc-hom-strict}
Does there exist a graph \(H\) such that
\[
    \delta_{\mathrm{hom}}^{\mathrm{VC}}(H)
    <
    \delta_{\mathrm{hom}}(H)?
\]
More generally, determine \(\delta_{\mathrm{hom}}^{\mathrm{VC}}(H)\) for graphs
\(H\) beyond cliques.
\end{problem}

\subsection{Asymmetric homomorphism thresholds under bounded VC-dimension}
\label{subsec:asym-vc-hom}

The clique case above reveals a tension between the chromatic and homomorphic
versions under bounded VC-dimension. On the one hand,
\[
    \delta_\chi^{\mathrm{VC}}(K_s)
    =
    \frac{s-3}{s-2},
\]
while on the other hand
\[
    \delta_{\mathrm{hom}}^{\mathrm{VC}}(K_s)
    =
    \frac{2s-5}{2s-3}.
\]
Thus after imposing bounded VC-dimension, there is a gap between
bounded colourability and bounded \(K_s\)-free homomorphic images.

This raises a natural interpolation question: can one describe, in a smooth
way, how the threshold decreases as the condition on the bounded target is
gradually weakened? Motivated by the asymmetric homomorphism threshold discussed
in Section~\ref{subsubsec:asymmetric-hom-thresholds}, it is natural to introduce
a VC-bounded asymmetric version.

Let \(\mathcal F_1\) and \(\mathcal F_2\) be two families of graphs. Call a real
number \(\alpha\ge 0\) \emph{VC-asymmetrically admissible} for
\((\mathcal F_1;\mathcal F_2)\) if, for every \(d\in\mathbb N\), there exists
an \(\mathcal F_2\)-free graph
\(\Gamma=\Gamma(\alpha,\mathcal F_1,\mathcal F_2,d)\) such that every \(n\)-vertex
\(\mathcal F_1\)-free graph \(G\) with \(\mathrm{VC}(G)\le d\) and
\(\delta(G)\ge \alpha n\) satisfies \(G\xrightarrow{\mathrm{hom}}\Gamma\). Define
\[
    \delta_{\mathrm{hom}}^{\mathrm{VC}}(\mathcal F_1;\mathcal F_2)
    :=
    \inf\{\alpha\ge 0:\alpha \text{ is VC-asymmetrically admissible for }
    (\mathcal F_1;\mathcal F_2)\}.
\]
When \(\mathcal F_i=\{H_i\}\), we write
\[
    \delta_{\mathrm{hom}}^{\mathrm{VC}}(H_1;H_2)
    :=
    \delta_{\mathrm{hom}}^{\mathrm{VC}}(\{H_1\};\{H_2\}).
\]

For cliques, this parameter interpolates between the VC-bounded homomorphism
threshold and the VC-bounded chromatic threshold. Indeed,
\[
    \delta_{\mathrm{hom}}^{\mathrm{VC}}(K_s;K_s)
    =
    \delta_{\mathrm{hom}}^{\mathrm{VC}}(K_s),
\]
whereas allowing the target to be \(K_t\)-free for larger and larger \(t\)
moves the problem closer to bounded colourability. One therefore expects a
chain of thresholds
\[
    \delta_{\mathrm{hom}}^{\mathrm{VC}}(K_s;K_s)
    \ge
    \delta_{\mathrm{hom}}^{\mathrm{VC}}(K_s;K_{s+1})
    \ge
    \delta_{\mathrm{hom}}^{\mathrm{VC}}(K_s;K_{s+2})
    \ge \cdots
\]
whose limiting value is
\[
    \delta_\chi^{\mathrm{VC}}(K_s)
    =
    \lim_{t\to\infty}
    \delta_{\mathrm{hom}}^{\mathrm{VC}}(K_s;K_t).
\]
The authors together with Rong and Xu~\cite{huang_interpolating_2025} determined
each step of this interpolation exactly.

\begin{theorem}[Huang--Liu--Rong--Xu~\cite{huang_interpolating_2025}]
\label{thm:vc-asym-cliques}
For all integers \(t\ge s\ge 3\),
\[
    \delta_{\mathrm{hom}}^{\mathrm{VC}}(K_s;K_t)
    =
    \frac{(s-3)(t-s+2)+1}{(s-2)(t-s+2)+1}.
\]
\end{theorem}

Moreover, the authors together with Rong and Xu~\cite{huang_interpolating_2025}
proved Conjecture~\ref{conj:asymmetric-hom-odd-cycles} under the additional
assumption of bounded VC-dimension.

\begin{theorem}[Huang--Liu--Rong--Xu~\cite{huang_interpolating_2025}]
\label{thm:vc-asym-odd-cycles}
For every integer \(t\ge 1\),
\[
    \delta_{\mathrm{hom}}^{\mathrm{VC}}
    (\mathcal C_{2t+3};\mathcal C_{2t+1})
    =
    0.
\]
\end{theorem}

\begin{remark}
The same work proves a slightly stronger statement:
\[
    \delta_{\mathrm{hom}}^{\mathrm{VC}}
    (\{C_{2t+3},C_{2t+1}\};C_{2t+1})
    =
    0.
\]
Thus the bounded-VC assumption verifies a weak version of the asymmetric
odd-cycle conjecture even when the source is only required to forbid the two
cycles \(C_{2t+3}\) and \(C_{2t+1}\), rather than all shorter odd cycles.
\end{remark}

\subsection{The VC threshold}
\label{subsec:vc-threshold}

In the preceding subsections, bounded VC-dimension was imposed as an additional
hypothesis. Under this hypothesis, chromatic thresholds can decrease
substantially, while homomorphism thresholds may behave differently: for
cliques, bounded VC-dimension lowers the chromatic threshold but does not lower
the homomorphism threshold. This suggests that VC-dimension can be a genuine
source of obstruction in minimum-degree threshold problems.

It is therefore natural to ask the complementary question: when does a
minimum-degree condition itself force bounded VC-dimension? In other words,
instead of assuming that the neighbourhood set system has bounded complexity,
we ask when bounded neighbourhood complexity is a consequence of density and
forbidden-subgraph structure.

For this question one needs a maximality assumption. VC-dimension is highly
sensitive to edge deletion: by deleting edges from a dense \(H\)-free graph, one
can create many different neighbourhood traces while preserving \(H\)-freeness
and, in many cases, linear minimum degree. Maximality rules out this artificial
source of complexity.

Recall that an \(H\)-free graph \(G\) is \emph{maximal \(H\)-free} if adding any
missing edge creates a copy of \(H\). The authors together with Rong and
Xu~\cite{huang_interpolating_2025} introduced the following threshold.

\begin{definition}[VC threshold]\label{def:vc-threshold}
For a fixed graph \(H\), define \(\delta_{\mathrm{VC}}(H)\) to be the infimum
over all \(\alpha\ge 0\) for which there exists \(d=d(\alpha,H)\) such that
every \(n\)-vertex maximal \(H\)-free graph \(G\) with \(\delta(G)\ge \alpha n\)
satisfies \(\mathrm{VC}(G)\le d\).
\end{definition}

Thus \(\delta_{\mathrm{VC}}(H)\) measures the minimum-degree density above which
maximal \(H\)-free graphs are forced to have bounded neighbourhood complexity.

In this terminology, the VC-dimension reduction of
{\L}uczak--Thomass\'e~\cite{luczak_coloring_2010} already gives the first basic
example:
$\delta_{\mathrm{VC}}(K_3)\le 1/3$.
More precisely, in the maximal triangle-free setting the VC-dimension method gives bounded neighbourhood VC-dimension whenever
\(\delta(G)\ge (1/3+\varepsilon)|V(G)|\).
This should be compared with the classical value
\(\delta_\chi(K_3)=1/3\): in the triangle case, bounded VC-dimension is forced
at exactly the density where bounded colourability first becomes unavoidable.

The authors together with Rong and Xu~\cite{huang_interpolating_2025} extended
this upper-bound phenomenon from triangles to all odd cycles. More precisely,
they proved that for every \(k\ge 3\),
\[
    \delta_{\mathrm{VC}}(C_{2k-1})
    \le
    \frac{1}{2k-1}.
\]
The matching lower bound was later obtained by Gishboliner and the
authors~\cite{gishboliner_huang_liu_vc_thresholds}. Thus the VC threshold of a
single odd cycle is now known exactly.

\begin{theorem}[Huang--Liu--Rong--Xu~\cite{huang_interpolating_2025};
Gishboliner--Huang--Liu~\cite{gishboliner_huang_liu_vc_thresholds}]
\label{thm:vc-odd-cycles}
For every \(k\ge 2\),
\[
    \delta_{\mathrm{VC}}(C_{2k-1})
    =
    \frac{1}{2k-1}.
\]
\end{theorem}

This should be compared with the chromatic and homomorphism thresholds. For
\(k\ge 3\), Thomassen's theorem gives
$\delta_\chi(C_{2k-1})=0$,
whereas the exact value of \(\delta_{\mathrm{hom}}(C_{2k-1})\) remains open.
Thus, for odd cycles, the VC threshold is strictly stronger than the chromatic
threshold and sits at the same scale as the best known homomorphism-threshold
upper bound.

Gishboliner and the authors~\cite{gishboliner_huang_liu_vc_thresholds} further
studied the general behaviour of the VC threshold. Their first result shows
that, unlike the chromatic threshold, the VC threshold never attains the
near-acyclic value.

\begin{theorem}[Gishboliner--Huang--Liu~\cite{gishboliner_huang_liu_vc_thresholds}]
\label{thm:vc-general-lower-bound}
Let \(H\) be a graph with \(\chi(H)=r\ge 3\). Then
\[
    \delta_{\mathrm{VC}}(H)
    >
    \frac{r-3}{r-2}.
\]
\end{theorem}

This is already a sharp qualitative contrast with the chromatic-threshold
classification. In the classical theory, the value $\frac{r-3}{r-2}$
is attained precisely by the \(r\)-near-acyclic graphs ( see \cite{ALLEN2013261} for definitions). Theorem~\ref{thm:vc-general-lower-bound}
says that no \(r\)-chromatic graph has VC threshold as small as this value.
Thus even when the chromatic threshold is minimal among \(r\)-chromatic graphs,
maximal \(H\)-free graphs still require a larger minimum-degree condition before
their neighbourhood complexity is forced to be bounded.

The VC threshold also interacts with the VC-bounded chromatic threshold through
the following simple inequality. For every graph \(H\),
\[
    \delta_\chi(H)
    \le
    \max\left\{
        \delta_{\mathrm{VC}}(H),
        \delta_\chi^{\mathrm{VC}}(H)
    \right\}.
\]
Indeed, given an \(H\)-free graph \(G\), one may extend it to a maximal
\(H\)-free graph \(G^+\) on the same vertex set. If the minimum degree is above
\(\delta_{\mathrm{VC}}(H)\), then \(G^+\) has bounded VC-dimension; applying
the VC-bounded chromatic threshold to \(G^+\) gives a bounded colouring, which
also restricts to \(G\).

Combining this inequality with Theorem~\ref{thm:vc-bounded-chromatic-general}
gives the following consequence. If \(\chi(H)=r\ge 3\) and
\[
    \delta_\chi(H)
    \le
    \frac{2r-5}{2r-3},
\]
then
$\delta_{\mathrm{VC}}(H)\ge \delta_\chi(H)$.
Indeed, in this range the VC-bounded chromatic threshold is at most
\((r-3)/(r-2)\), while Theorem~\ref{thm:vc-general-lower-bound} rules out
\(\delta_{\mathrm{VC}}(H)\le (r-3)/(r-2)\). In the intermediate case
\(\delta_\chi(H)=(2r-5)/(2r-3)\), the displayed inequality above forces
\(\delta_{\mathrm{VC}}(H)\ge \delta_\chi(H)\).

This suggests a natural problem asked in \cite{gishboliner_huang_liu_vc_thresholds}.

\begin{problem}\label{prob:vc-vs-chromatic}
Is it true that
$\delta_{\mathrm{VC}}(H)\ge \delta_\chi(H)$
for every graph \(H\) with \(\chi(H)\ge 3\)?
More strongly, is it true that
$\delta_{\mathrm{VC}}(H)\ge \delta_{\mathrm{hom}}(H)$
for every graph \(H\) with \(\chi(H)\ge 3\)?
\end{problem}

Gishboliner and the authors also determined the VC threshold for all complete
tripartite graphs.

\begin{theorem}[Gishboliner--Huang--Liu~\cite{gishboliner_huang_liu_vc_thresholds}]
\label{thm:vc-complete-tripartite}
Let \(1\le a\le b\le c\). Then
\[
    \delta_{\mathrm{VC}}(K_{a,b,c})
    =
    \begin{cases}
        \dfrac13, & \text{if } a=b=1,\\[6pt]
        \dfrac{b}{2b+1}, & \text{if } a=1,\ b\ge 2,\ \text{and } c\ge 2b-2,\\[6pt]
        \dfrac12, & \text{otherwise.}
    \end{cases}
\]
\end{theorem}

This theorem shows that, even within 3-chromatic graphs, the VC threshold has a
richer behaviour than the chromatic threshold.
To make this spectral viewpoint explicit, we similarly define
\[
    \Delta_{\mathrm{VC}}
    :=
    \{\delta_{\mathrm{VC}}(H):\chi(H)\ge 3\},
\]
and, for each integer \(r\ge 3\),
\[
    \Delta_{\mathrm{VC}}^{(r)}
    :=
    \{\delta_{\mathrm{VC}}(H):\chi(H)=r\}.
\]
Thus \(\Delta_{\mathrm{VC}}\) is the non-bipartite spectrum of the VC threshold,
and \(\Delta_{\mathrm{VC}}^{(r)}\) is its restriction to \(r\)-chromatic
forbidden graphs.
We say that a value \(\alpha\) is an \emph{accumulation point} of
\(\Delta_{\mathrm{VC}}^{(r)}\) if there exists a sequence of \(r\)-chromatic
graphs \(\{H_n\}_{n\ge 1}\) such that the values
\(\delta_{\mathrm{VC}}(H_n)\) are pairwise distinct and
\[
    \lim_{n\to\infty}\delta_{\mathrm{VC}}(H_n)=\alpha.
\]
Theorem~\ref{thm:vc-complete-tripartite} already implies that
\(\Delta_{\mathrm{VC}}^{(3)}\) has an accumulation point. Indeed, the distinct
values
\[
    \delta_{\mathrm{VC}}(K_{1,b,2b-2})
    =
    \frac{b}{2b+1}
    \qquad (b\ge 2)
\]
converge to \(1/2\) as \(b\to\infty\). Hence \(1/2\) is an accumulation point of
\(\Delta_{\mathrm{VC}}^{(3)}\).

Moreover, Theorem~\ref{thm:vc-complete-tripartite} also implies that $\delta_{\mathrm{VC}}$ is \emph{not monotone} under taking induced subgraphs.
To see this, by Theorem~\ref{thm:vc-complete-tripartite}, we have 
\begin{align*}
    K_{1,3,3}\subseteq K_{1,3,10} \quad\text{where}\quad
     \delta_{\mathrm{VC}}(K_{1,3,3})= \frac{1}{2}> \frac{3}{7} = \delta_{\mathrm{VC}}(K_{1,3,10}).
\end{align*}

These phenomena show that the VC threshold is substantially less rigid than the
chromatic threshold. For chromatic thresholds, the fixed-colour spectra are
finite and the parameter is monotone under taking subgraphs. In contrast,
\(\Delta_{\mathrm{VC}}^{(3)}\) already has an accumulation point, and
\(\delta_{\mathrm{VC}}\) fails monotonicity even under taking induced subgraphs.
Thus bounded neighbourhood complexity is governed by finer structural features
of the forbidden graph than those detected by the chromatic threshold alone.
Determining \(\delta_{\mathrm{VC}}(H)\) in general therefore seems to require
new ideas beyond the monotonicity-based framework of the classical chromatic
threshold classification.

\begin{problem}\label{prob:vc-more-values}
Determine $\Delta_{\mathrm{VC}}^{(3)}$.
\end{problem}



\section{Blowup thresholds}\label{sec:blowup-threshold}
In this section, we introduce a strengthening of both the homomorphism threshold
and the VC threshold.

Recall first that a homomorphism to a bounded graph gives only a coarse
structural description: it says that \(G\) is a subgraph of a blowup of the
image, but the pairs of fibres corresponding to edges of the template may have
been arbitrarily thinned. If \(G\) is maximal \(H\)-free, however, it is natural
to hope that such thinning is impossible, so that the relevant pairs of fibres
must in fact be complete. This is exactly what happens for cliques. In the
clique case, the known proofs of the homomorphism threshold first pass to a
maximal \(K_r\)-free graph; after maximal completion, the bounded homomorphic
description upgrades to an actual blowup of a bounded \(K_r\)-free graph
\cite{oberkampf_structure_2020}. This raises the natural question of whether a
similar strengthening holds for more general forbidden graphs.

There is a second, closely related motivation. Already in {\L}uczak's regularity
proof for the triangle case, maximality leads to a bounded partition which is
almost homogeneous, in the sense that most pairs of parts have density close to
either \(0\) or \(1\). Thus one may ask more generally when an almost homogeneous
bounded partition can be promoted to a genuinely homogeneous one, equivalently
to a blowup of a bounded graph. Such a strengthening would be particularly
useful for graphs of bounded VC-dimension, where bounded-VC regularity lemmas
already provide bounded almost homogeneous partitions
\cite{AlonFischerNewman,LovaszSzegedy}. From this viewpoint, the main problem
is to understand when approximate homogeneous structure can be made exact.

This motivates the following strengthening, recently introduced
in~\cite{huang_interpolating_2025}. For a graph \(F\), a blowup of \(F\) is
obtained by replacing every vertex of \(F\) by an independent set, and every
edge of \(F\) by a complete bipartite graph between the corresponding parts.
We write \(F[\cdot]\) for an arbitrary blowup of \(F\)\footnote{Here we allow
the independent set replacing a vertex to be empty.}. The \emph{blowup
threshold} of \(H\) is defined as
\begin{align*}
    \delta_{\textup{B}}(H):=\inf\big\{&\alpha\ge 0:\exists~H\textup{-free graph }F=F(\alpha,H)\  \textup{s.t.~}
    \forall~\textup{maximal\ }H\textup{-free}~G,\ \\&\delta(G)\ge \alpha |V(G)|
    \Rightarrow G=F[\cdot]\big\}.
\end{align*}

The first fundamental case is again given by cliques. In this case the blowup
threshold coincides with the chromatic and homomorphism thresholds.

\begin{theorem}[Oberkampf--Schacht~\cite{oberkampf_structure_2020}]
\label{thm:blowup-cliques}
For every \(s\ge 3\),
\[
    \delta_{\mathrm{B}}(K_s)
    =
    \delta_{\mathrm{hom}}(K_s)
    =
    \delta_{\chi}(K_s)
    =
    \frac{2s-5}{2s-3}.
\]
\end{theorem}

The clique case is therefore optimistic: the exact blowup conclusion costs no
more minimum degree than bounded colourability or a bounded homomorphic image.

The next fundamental family is given by odd cycles. Here the behaviour is very
different from the chromatic threshold. Recall that for every \(k\ge 3\),
$\delta_\chi(C_{2k-1})=0$.
For the homomorphism threshold, the exact value of
\(\delta_{\mathrm{hom}}(C_{2k-1})\) is still unknown. The blowup threshold,
however, can be determined exactly.

\begin{theorem}[Huang--Liu--Rong--Xu~\cite{huang_interpolating_2025}]
\label{thm:blowup-odd-cycles}
For every \(k\ge 3\),
\[
    \delta_{\mathrm{B}}(C_{2k-1})
    =
    \frac{1}{2k-1}.
\]
\end{theorem}

Thus odd cycles already show that the exact blowup conclusion is substantially
stronger than bounded colourability:
\[
    \delta_\chi(C_{2k-1})
    =
    0
    <
    \frac{1}{2k-1}
    =
    \delta_{\mathrm{B}}(C_{2k-1}).
\]
At the same time, the value agrees with the best known upper bound for the
homomorphism threshold of \(C_{2k-1}\). In this sense, the blowup threshold gives
an exact structural answer in a regime where the corresponding homomorphism
threshold remains difficult to determine.

The authors together with Rong~\cite{huang_liu_rong_blowup_thresholds} further showed that the blowup threshold differs
from the chromatic threshold in several more fundamental ways. The first
difference is spectral: the blowup threshold realizes new exact values already
among 3-chromatic forbidden graphs.

\begin{theorem}[Huang--Liu--Rong~\cite{huang_liu_rong_blowup_thresholds}]
\label{thm:blowup-new-values}
For every integer \(k\ge 2\), there exists a graph \(H_k\) with
\(\chi(H_k)=3\) such that
\[
    \delta_{\mathrm{B}}(H_k)=\frac{1}{2k}.
\]
\end{theorem}

For \(k\ge 2\), the values \(1/(2k)\) do not occur as chromatic thresholds.
Moreover, no corresponding examples are currently known for homomorphism
thresholds. Thus the blowup threshold produces genuinely new exact values in
the 3-chromatic regime.

A second difference is that the blowup threshold never vanishes for
non-bipartite forbidden graphs.

\begin{theorem}[Huang--Liu--Rong~\cite{huang_liu_rong_blowup_thresholds}]
\label{thm:blowup-positive}
For every graph \(H\) with \(\chi(H)\ge 3\),
\[
    \delta_{\mathrm{B}}(H)>0.
\]
\end{theorem}

\begin{remark}
The case \(\chi(H)\ge 4\) follows immediately from
\(\delta_{\mathrm{B}}(H)\ge \delta_\chi(H)>0\). The main point is therefore the
3-chromatic case, where the chromatic threshold may vanish. This positivity can
also be deduced from Theorem~\ref{thm:vc-general-lower-bound}, since a blowup of a bounded graph has bounded
VC-dimension, and hence
$\delta_{\mathrm{B}}(H)\ge \delta_{\mathrm{VC}}(H)>0$ when $\chi(H)=3$.
\end{remark}

The third difference is structural. Unlike the chromatic threshold, the blowup
threshold is not monotone as a graph parameter.

\begin{theorem}[Huang--Liu--Rong~\cite{huang_liu_rong_blowup_thresholds}]
\label{thm:blowup-nonmonotone}
The blowup threshold is not monotone under taking induced subgraphs.
\end{theorem}

Thus the blowup threshold differs
from the chromatic threshold in three complementary ways: as a graph parameter
it is not monotone, as a spectrum it has new exact values, and as a complexity
measure it is richer even within the class of 3-chromatic forbidden graphs.
Consequently, the blowup threshold cannot be understood by a direct analogue of
the monotonicity-based strategy used for chromatic thresholds.

We end this section by recording several natural open problems about the
spectrum of blowup thresholds raised in \cite{huang_liu_rong_blowup_thresholds}.  
Recall that
\[
    \Delta_{\mathrm{B}}
    :=
    \{\delta_{\mathrm{B}}(H):\chi(H)\ge 3\},
\]
and, for each integer \(r\ge 3\),
\[
    \Delta_{\mathrm{B}}^{(r)}
    :=
    \{\delta_{\mathrm{B}}(H):\chi(H)=r\}.
\]
We say that a value \(\alpha\) is an \emph{accumulation point} of
\(\Delta_{\mathrm{B}}\) if there exists a sequence of graphs
\(\{H_n\}_{n\ge 1}\) such that the values \(\delta_{\mathrm{B}}(H_n)\) are
pairwise distinct and
\[
    \lim_{n\to\infty}\delta_{\mathrm{B}}(H_n)=\alpha.
\]

The exact values currently recorded above include
$\left\{\frac{1}{m}:m\ge 3\right\}$,
combining the odd-cycle values \(1/(2k-1)\), the clique value \(\delta_{\mathrm{B}}(K_3)=1/3\), and the even-denominator values
\(1/(2k)\) above. The value \(1/2\) also occurs in the examples of~\cite{huang_liu_rong_blowup_thresholds}, but it is not supplied by the preceding two displayed families. Hence \(0\) is an accumulation point of
\(\Delta_{\mathrm{B}}\). More generally, one may ask whether the lower endpoints
coming from the chromatic-threshold classification also occur as accumulation
points.

\begin{conj}\label{conj:blowup-accumulation-endpoints}
For every integer \(m\ge 2\), the value $\frac{m-1}{m}$
is an accumulation point of \(\Delta_{\mathrm{B}}\). 
\end{conj}

A stronger possibility is that the blowup-threshold spectrum is dense in a
nontrivial interval.

\begin{conj}\label{conj:blowup-spectrum-interval}
There exists a non-degenerate interval \(I\subseteq (0,1]\) such that, for every
rational number \(q\in I\), there exists a graph \(H_q\) satisfying
\[
    \delta_{\mathrm{B}}(H_q)=q.
\]
\end{conj}

Even the 3-chromatic slice of the spectrum is far from understood. The new
values \(1/(2k)\), the odd-cycle values \(1/(2k-1)\), and the non-monotonicity
phenomenon all occur already for 3-chromatic forbidden graphs.

\begin{problem}\label{prob:blowup-spectrum-three-chromatic}
Determine the set
\[
    \Delta_{\mathrm{B}}^{(3)}
    =
    \{\delta_{\mathrm{B}}(H):\chi(H)=3\}.
\]
\end{problem}

Among possible values in \(\Delta_{\mathrm{B}}^{(3)}\), the value \(1/2\) is
especially natural: it is the largest possible chromatic threshold for
3-chromatic forbidden graphs, it occurs as a blowup threshold in~\cite{huang_liu_rong_blowup_thresholds}, and the non-monotonicity examples also appear around this value.

\begin{problem}\label{prob:blowup-half}
Determine necessary and sufficient conditions for a 3-chromatic graph \(H\) to
satisfy
\[
    \delta_{\mathrm{B}}(H)=\frac12.
\]
\end{problem}


\section{Minimum-degree removal lemma thresholds}\label{sec:removal}

A recurring philosophy behind the threshold problems in this survey is that \emph{density forces structure}.  For chromatic and homomorphism thresholds, one asks when a linear minimum-degree condition already forces a bounded colouring or a bounded homomorphic image.  A closely related robust version asks whether the same local density condition improves removal and supersaturation estimates.
A good example comes from the graph removal lemma, one of the central results in extremal graph theory with numerous applications (for instance to additive combinatorics and property testing); see~\cite{fox2023minimum,gishboliner2024minimum} for background and further references.
In its standard form, it asserts that for every fixed graph $H$ and every $\varepsilon>0$, there exists $\delta=\delta(\varepsilon,H)>0$ such that every $n$-vertex graph $G$ with fewer than $\delta n^{v(H)}$ copies of $H$ can be made $H$-free by removing at most $\varepsilon n^2$ edges.

While the statement is clean, the quantitative dependence of $\delta$ on $\varepsilon$ is notoriously poor: for many non-bipartite graphs $H$ the best known bounds are far from polynomial, and in the clique case the general upper bound coming from regularity-type proofs is of tower type in $\log(1/\varepsilon)$.
More broadly, Alon~\cite{alon2002testing} showed that the removal parameter $\delta(\varepsilon,H)$ is polynomial in $\varepsilon$ \emph{if and only if} $H$ is bipartite, so without any additional assumptions one cannot hope for polynomial bounds in the non-bipartite case.
This motivates the following question asked by Fox and Wigderson~\cite{fox2023minimum}.
\begin{question}\label{ques: removal threshold}
    Does imposing a minimum-degree condition on the host graph $G$ improve the removal lemma bounds?
\end{question}

Following Fox and Wigderson~\cite{fox2023minimum}, it is convenient to work with a supersaturation-type formulation.
For $\gamma\in[0,1]$, define $\delta(\varepsilon,H;\gamma)$ to be the maximum $\delta\in[0,1]$ with the property that whenever $G$ is an $n$-vertex graph with
\[
\delta(G)\ge \gamma n
\quad\text{and}\quad
\text{$G$ contains at least $\varepsilon n^2$ edge-disjoint copies of $H$,}
\]
then $G$ contains at least $\delta\, n^{v(H)}$  copies of $H$.
Thus the usual removal lemma corresponds to $\gamma=0$, i.e. no minimum-degree assumption.

The central point is that, above an appropriate minimum-degree proportion, one can force \emph{polynomial} (or even \emph{linear}) dependence of $\delta(\varepsilon,H;\gamma)$ on $\varepsilon$.
This leads to two threshold parameters.

\begin{definition}[Minimum-degree removal thresholds~\cite{fox2023minimum}]
Let $H$ be a fixed graph.
\begin{itemize}
\item The \emph{linear removal threshold} of $H$, denoted $\delta_{\mathrm{lin\text{-}rem}}(H)$, is the infimum over all $\gamma$ such that there exists $\mu=\mu(\gamma,H)>0$ with
\[
\delta(\varepsilon,H;\gamma)\ge \mu\,\varepsilon
\qquad\text{for all }\varepsilon>0.
\]
\item The \emph{polynomial removal threshold} of $H$, denoted $\delta_{\mathrm{poly\text{-}rem}}(H)$, is the infimum over all $\gamma$ such that there exists $\mu=\mu(\gamma,H)>0$ with
\[
\delta(\varepsilon,H;\gamma)\ge \mu\,\varepsilon^{1/\mu}
\qquad\text{for all }\varepsilon>0.
\]
\end{itemize}
In particular, $\delta_{\mathrm{lin\text{-}rem}}(H)\ge \delta_{\mathrm{poly\text{-}rem}}(H)$ holds trivially.
\end{definition}

Fox and Wigderson~\cite{fox2023minimum} initiated the systematic study of these thresholds and determined them for cliques.

\begin{theorem}[Fox, Wigderson~\cite{fox2023minimum}]\label{thm:FW-clique-removal-threshold}
For every $r\ge 3$,
\[
\delta_{\mathrm{lin\text{-}rem}}(K_r)=\delta_{\mathrm{poly\text{-}rem}}(K_r)=\frac{2r-5}{2r-3}.
\]
Moreover, above this threshold one obtains linear bounds in the $K_r$ removal lemma, while below it the bounds are again super-polynomial (as in the unrestricted clique removal lemma).
\end{theorem}

Very recently, Gishboliner, Jin and Sudakov~\cite{gishboliner2024minimum} determined these thresholds for odd cycles (and more generally gave a complete description of $\delta_{\mathrm{lin\text{-}rem}}(H)$ for all $3$-chromatic $H$).

\begin{theorem}[Gishboliner, Jin, Sudakov~\cite{gishboliner2024minimum}]\label{thm:GJS-oddcycle-removal-threshold}
For every integer $k\ge 1$,
\[
\delta_{\mathrm{poly\text{-}rem}}(C_{2k+1})=\frac{1}{2k+1}.
\]
Furthermore, if $\chi(H)=3$ then
\[
\delta_{\mathrm{lin\text{-}rem}}(H)=
\begin{cases}
\frac12, & \text{$H$ has no critical edge},\\[2pt]
\frac13, & \text{$H$ has a critical edge and contains a triangle},\\[2pt]
\frac14, & \text{$H$ has a critical edge and odd-girth$(H)\ge 5$.}
\end{cases}
\]
\end{theorem}

These results illustrate a striking ``threshold'' phenomenon parallel to chromatic and homomorphism thresholds:
once the minimum degree crosses a critical value, one can often replace the tower-type (or worse) quantitative behaviour in a general extremal lemma by much more efficient, polynomial-type bounds.

Although $\delta_{\mathrm{lin\text{-}rem}}(H)$ is now completely determined when $\chi(H)=3$, the polynomial threshold $\delta_{\mathrm{poly\text{-}rem}}(H)$ is known in far fewer cases beyond odd cycles.
It would be interesting to ask the following problems.

\begin{problem}\label{prob: more values on poly-rem}
    Do there exist more values beyond $\{\frac{1}{2k-1}\}_{k\ge2}$ for $\delta_{\mathrm{poly\text{-}rem}}(H)$ when $\chi(H) = 3$?
\end{problem}\label{ques: more values on poly-rem}

\begin{problem}\label{prob: determine larger chi for lin-rem}
    Determine $\delta_{\mathrm{lin\text{-}rem}}(H)$ for $\chi(H)\ge 4$.
\end{problem}

Fox and Wigderson~\cite{fox2023minimum} also proposed several broader problems
connecting minimum-degree removal thresholds with the chromatic and
homomorphism thresholds.
They suggest that the chromatic, homomorphism,
linear-removal and polynomial-removal thresholds should not satisfy any general
ordering beyond the tautological inequalities
\[
    \delta_\chi(H)\le \delta_{\mathrm{hom}}(H)
    \qquad\text{and}\qquad
    \delta_{\mathrm{poly\text{-}rem}}(H)
    \le
    \delta_{\mathrm{lin\text{-}rem}}(H).
\]

\begin{conj}[Fox--Wigderson~\cite{fox2023minimum}]
\label{conj:four-thresholds-distinct}
There exists a graph \(H\) for which the four quantities
\[
    \delta_\chi(H),\quad
    \delta_{\mathrm{hom}}(H),\quad
    \delta_{\mathrm{lin\text{-}rem}}(H),\quad
    \delta_{\mathrm{poly\text{-}rem}}(H)
\]
are all distinct.
\end{conj}



\section{Further connections and future directions}
\label{sec:conclusion}

The threshold parameters surveyed above are best viewed as different ways in
which a local density assumption can force global bounded structure.  A useful
schematic list is
\[
    \delta_\chi,\qquad
    \delta_{\mathrm{hom}},\qquad
    \delta_{\mathrm{VC}},\qquad
    \delta_{\mathrm{B}},\qquad
    \delta_{\mathrm{poly\text{-}rem}},\qquad
    \delta_{\mathrm{lin\text{-}rem}}.
\]
The chromatic threshold asks for bounded colourability; the homomorphism
threshold asks for a bounded forbidden-template image; the VC and blowup
thresholds ask for bounded neighbourhood complexity or exact finite-template
structure; and the removal thresholds ask for improved quantitative
supersaturation.  The examples in this survey suggest that this hierarchy is
conceptual rather than monotone: apart from the formal inequalities
\[
\delta_\chi(H)\le \delta_{\mathrm{hom}}(H)
\le \delta_{\mathrm{B}}(H),\ 
\delta_{\mathrm{VC}}(H)\le \delta_{\mathrm{B}}(H),
\ 
\delta_{\mathrm{poly\text{-}rem}}(H)\le
\delta_{\mathrm{lin\text{-}rem}}(H).
\]
there seems to be no universal ordering without additional hypotheses.

\paragraph{Approximate templates and removal.}
Approximate homomorphisms provide a natural bridge between homomorphism
thresholds and removal thresholds.  Fox and Zhao~\cite{fox_zhao_approx_homomorphisms_2022}
showed that every triangle-free graph admits an \(\varepsilon\)-approximate
homomorphism to a bounded triangle-free graph, where only
\(O(\varepsilon n^2)\) edges are allowed to fail.  From this viewpoint,
minimum-degree homomorphism thresholds ask when such an approximate
bounded-template description can be made exact.  The asymmetric approximate
homomorphism results of Gishboliner--Hurley--Wigderson~\cite{gishboliner_hurley_wigderson_asymmetric_2025}
point in the same direction.

A related local certificate appears in the work of Fox and
Wigderson~\cite{fox2023minimum}.  They define the \emph{popular edge threshold}
\(\delta_{\mathrm{pop\text{-}edge}}(H)\) as the infimum of all \(\gamma\) such
that, for some \(\beta>0\), every copy of \(H\) in every graph
\(G\) with \(\delta(G)\ge \gamma |V(G)|\) contains an edge lying in at least
\(\beta |V(G)|^{v(H)-2}\) copies of \(H\).  This condition is strong enough to
imply linear removal bounds, and hence
\[
    \delta_{\mathrm{lin\text{-}rem}}(H)
    \le
    \delta_{\mathrm{pop\text{-}edge}}(H).
\]
It would be interesting to understand when this sufficient condition is also
close to necessary.

\paragraph{Second-order density and popular non-edges.}
The proof mechanisms behind the homomorphism, VC and blowup thresholds often use
more than ordinary minimum degree.  For instance, in a maximal triangle-free
graph above the \(1/3\) threshold, every non-edge has a large common
neighbourhood.  Wang--Yang--Zhao~\cite{wang_yang_zhao_minimum_common_degree}
showed that this second-order condition has independent force: if a
triangle-free graph \(G\) on \(n\) vertices satisfies
\[
    \delta_2^\circ(G):=
    \min_{xy\notin E(G)} |N_G(x)\cap N_G(y)|
    >\left\lfloor \frac n8\right\rfloor,
\]
then \(G\) is homomorphic to \(C_5\), strengthening H\"aggkvist's theorem.

There is an edge-side analogue in the theory of books and joints.  The booksize
of a graph is the maximum number of triangles sharing a common edge, and more
generally one may ask how many copies of a fixed graph can share a common edge.
Fox and Loh~\cite{fox_loh_edges_triangles} studied a problem of
Erd\H{o}s and Rothschild on dense graphs in which every edge lies in a triangle,
showing that below the Tur\'an density \(1/4\) one cannot in general force a
polynomial-size book.  This suggests a dual viewpoint for maximal \(H\)-free
graphs: popular edges in removal problems should be compared with \emph{popular
non-edges}.  If \(uv\notin E(G)\), let \(\mathrm{comp}_H^G(u,v)\) denote the
number of labelled ways in which adding \(uv\) completes a copy of \(H\).  Then
maximal \(H\)-freeness is the qualitative condition
\(\mathrm{comp}_H^G(u,v)>0\) for every non-edge, while a quantitative analogue
would ask for
\[
    \mathrm{comp}_H^G(u,v)
    \ge \eta |V(G)|^{v(H)-2}
    \qquad\text{for every }uv\notin E(G).
\]
This leads to a possible ``popular non-edge'' threshold for forcing bounded
homomorphic images, bounded VC-dimension, or exact blowup structure.

\paragraph{Robust second-order and trigraph complexity.}
The exact blowup arguments suggest robust variants in which a small set of
non-edges is allowed to have poor completion behaviour.  For example, in the
triangle-free case one may define
\[
    D_\tau(G):=
    \{xy\notin E(G): |N_G(x)\cap N_G(y)|<\tau |V(G)|\}.
\]
The case \(D_\tau(G)=\emptyset\) is the ideal second-order hypothesis; a robust
version would ask whether sparsity, bounded chromatic number, or bounded
complexity of \(D_\tau(G)\) still forces approximate or exact bounded-template
structure.  This is closely related to the Dense Neighborhood Lemma of
Bourneuf--Charbit--Thomass\'e~\cite{bourneuf_charbit_thomasse_dnl_2025}, whose
partial-concept-class and trigraph framework is designed to handle situations
where the ordinary VC-dimension may be unbounded but the relevant defective
relation is controlled.  A natural refinement is therefore to replace the usual
neighbourhood system \(\{N_G(v):v\in V(G)\}\) by a second-order system such as
\[
    \mathcal N_2^\circ(G)
    :=\{N_G(x)\cap N_G(y):xy\notin E(G)\},
\]
and to ask whether bounded complexity of this system suffices for the packing
and exactification arguments used in the blowup-threshold theory.

\paragraph{Stability, spectra and neighbouring template theories.}
Several adjacent theories suggest further problems.  Minimum-degree stability,
studied by Illingworth~\cite{illingworth_min_degree_stability_2023}, asks when
large minimum degree forces an \(H\)-free graph to be \(o(n^2)\)-close to a
Tur\'an-type multipartite structure.  Recent stability results near the
chromatic threshold~\cite{kim_stability_2025} and the edge-density/minimum-degree
tradeoff of Wu and Xue~\cite{wu_xue_edge_density_min_degree_2025} indicate that
one should also seek stability theorems for \(\delta_{\mathrm{hom}}\),
\(\delta_{\mathrm{VC}}\), \(\delta_{\mathrm B}\), and the removal thresholds.
Finite-template viewpoints also appear in graph-limit language, where bounded blowups correspond to finite-step
graphons~\cite{LovaszSzegedy}.  On the topological side, Sankar's odd-cycle
obstructions are naturally connected with the broader theory of Hom-complexes
\cite{babson_kozlov_hom_complexes_2006}.  Finally, coloured and arithmetic
analogues, such as the two-colour threshold problem of
Gao--Liu--Wu--Xue~\cite{gao_liu_wu_xue_chromatic_pairs_2026} and chromatic
thresholds for linear equations~\cite{liu2026chromaticthresholdslinearequations},
suggest that the principle ``local density forces bounded complexity'' extends
beyond ordinary forbidden-subgraph problems.

We close with a short list of directions which seem especially close to the
thresholds discussed in this survey.

\begin{problem}[Exactifying approximate templates]
Determine conditions under which an approximate homomorphism from an \(H\)-free
graph to a bounded \(H\)-free graph can be upgraded to an exact homomorphism.
In particular, can minimum-degree assumptions, popular non-edge conditions, or
lower bounds on \(\mathrm{comp}_H^G(u,v)\) force such an upgrade?
\end{problem}

\begin{problem}[Popular non-edge thresholds]
For which graphs \(H\) does a positive lower bound on
\[
    \min_{uv\notin E(G)}\mathrm{comp}_H^G(u,v)
\]
force bounded homomorphic images, bounded VC-dimension, or bounded blowup
structure for maximal \(H\)-free graphs?  How does the resulting threshold
compare with \(\delta_{\mathrm{hom}}(H)\), \(\delta_{\mathrm{VC}}(H)\), and
\(\delta_{\mathrm B}(H)\)?
\end{problem}

\begin{problem}[Robust second-order blowup]
Let \(G\) be a maximal triangle-free graph and let
\(D_\tau(G)=\{xy\notin E(G): |N_G(x)\cap N_G(y)|<\tau |V(G)|\}\).  If
\(|D_\tau(G)|=o(n^2)\), or if \(D_\tau(G)\) has bounded trigraph complexity, must
\(G\) be close to a blowup of a bounded triangle-free graph?  More generally,
can ordinary VC-dimension in the current blowup arguments be replaced by a
second-order or trigraph VC-dimension?
\end{problem}

\begin{problem}[Spectra and comparison]
Determine the spectra
\[
    \Delta_{\mathrm{hom}}^{(r)},\qquad
    \Delta_{\mathrm{VC}}^{(r)},\qquad
    \Delta_{\mathrm B}^{(r)}
\]
for fixed \(r\ge 3\), and decide which inequalities between the homomorphism,
VC, blowup, polynomial-removal and linear-removal thresholds can hold.  In
particular, are there intervals or dense rational subsets in any of these
spectra, and can one find examples where all non-tautologically related
thresholds are distinct?
\end{problem}
\bibliographystyle{abbrv}
\bibliography{Survey_Homomorphism_updated}


%

\end{document}